\documentclass[11pt,twoside,]{amsart}
\usepackage[dvips]{graphicx}
\usepackage{tikz}
\usepackage{amsmath, amsthm, amscd, amsfonts, amssymb, color}
\newtheorem{thm}{Theorem}[section]
\newtheorem{cor}[thm]{Corollary}
\newtheorem{lem}[thm]{Lemma}
\newtheorem{prop}[thm]{Proposition}
\newtheorem{defn}[thm]{Definition}
\newtheorem{rem}[thm]{Remark}

\newtheorem{qu}[thm]{Question}
\numberwithin{equation}{section}

\begin{document}

\oddsidemargin 0mm
\evensidemargin 0mm

\thispagestyle{plain}

\vspace{5cc}
\begin{center}

{\large\bf Groups all of whose undirected Cayley graphs are determined by their spectra}
\rule{0mm}{6mm}\renewcommand{\thefootnote}{}
\footnotetext{

{\rule{2.4mm}{0mm}Keywords and Phrases. Cayley graph; Spectrum of a graph; Graphs determined by spectrum}}

\vspace{1cc}
{\large\it Alireza Abdollahi, Shahrooz Janbaz and Mojtaba Jazaeri}

\vspace{1cc}
\parbox{24cc}{{\small

Let $G$ be a finite group, and $S$ be a subset of $G\setminus\{1\}$ such that $S=S^{-1}$. Suppose that $Cay(G,S)$ is the Cayley graph on $G$ with respect to the set $S$ which is the graph whose vertex set is $G$ and two vertices $a,b\in G$ are adjacent whenever $ab^{-1}\in S$. The adjacency spectrum $Spec(\Gamma)$ of a graph $\Gamma$ is the multiset of eigenvalues of its adjacency matrix. A graph $\Gamma$ is called ``determined by its spectrum" (or for short DS) whenever if a graph $\Gamma'$ has the same spectrum as $\Gamma$, then $\Gamma \cong \Gamma'$. We say that the group $G$ is DS (Cay-DS, respectively) whenever if $\Gamma$ is a Cayley graph over $G$ and $Spec(\Gamma)=Spec(\Gamma')$ for some graph (Cayley graph, respectively) $\Gamma'$, then $\Gamma \cong \Gamma'$. In this paper, we study finite DS groups and finite Cay-DS groups. In particular we prove that all finite DS groups are solvable and all  Sylow $p$-subgroups of a finite DS group is 
cyclic for all $p\geq 5$.  We also give several infinite families of non Cay-DS solvable groups. In particular we prove that there exist two cospectral non-isomorphic $6$-regular Cayley graphs on the dihedral group of order $2p$  for any prime $p\geq 13$. }}
\end{center}

\vspace{1cc}
\vspace{1cc}
\begin{center}
\section {\bf Introduction and Results}
\end{center}

Let $G$ be a finite group and $S$ be a subset of $G\setminus \{1\}$ such that $S=S^{-1}$. The Cayley graph $Cay(G,S)$ is the graph whose vertex set is $G$ and two vertices $a,b\in G$ are adjacent whenever $ab^{-1}\in S$. The adjacency spectrum $Spec(\Gamma)$ of a graph $\Gamma$ is the multiset of eigenvalues of its adjacency matrix. Two graphs $\Gamma$ and $\Gamma'$ are called cospectral if $Spec(\Gamma)=Spec(\Gamma')$ and $\Gamma'$ is called to be a cospectral mate for $\Gamma$. A graph $\Gamma$ is called determined by the spectrum, if $\Gamma$ is isomorphic to all its cospectral mates. Also, we say that the group $G$ is DS (Cay-DS, respectively) whenever if $\Gamma$ is a Cayley graph over $G$ and $Spec(\Gamma)=Spec(\Gamma')$ for some graph (Cayley graph, respectively) $\Gamma'$, then $\Gamma \cong \Gamma'$. In general a Cay-DS graph (which is a Cayley graph) is not necessarily a DS graph (see Proposition \ref{prop29}, below). The characterization of DS graphs is very difficult and goes back to about half of a century and it is originated in chemistry \cite{GHZ,VH}. For a survey of results on DS graphs one may see \cite{VH} and references there in. It is well-known that all regular graphs with at most $9$ vertices are DS \cite[page 398]{CD}. So all groups of order at most $9$ are DS. It is proved that no two non-isomorphic Cayley graphs on the same group of prime order are cospectral (see \cite{DJO}). By our terminology, this means that the cyclic group  $C_p$ of prime order $p$ is Cay-DS. In contrast, it is shown that $D_{2p}$ is not Cay-DS for any prime $p\geq 129$ (see \cite{Ba}).    It is  shown in \cite{LSV} that the projective special linear group $PSL_d(\mathbb{F}_q)$ of dimension $d$ over the field $\mathbb{F}_q$ with $q$ elements is not Cay-DS for certain values of $d$ and $q$.
In this paper we study the following questions.
\begin{qu}\label{question1} ~~\\
{\rm (1)} \; Which finite groups are Cay-DS? \\
{\rm (2)}  \;  Which finite groups are DS?
\end{qu}

Throughout the paper we use the following notations: $C_n$ denotes the cyclic group of order $n$; the dihedral group of order $2n$ is denoted by $D_{2n}$. 

Our main results are:

\begin{thm}\label{corp}
\begin{enumerate}
\item Every Sylow $p$-subgroup of any finite Cay-DS group is cyclic whenever $p>5$. Every Sylow $5$-subgroup of any finite DS group is cyclic.
\item Every  Sylow $2$-group of a Cay-DS group is of order at most $16$. 
\item Every Sylow $3$-subgroup of any finite Cay-DS group is either cyclic or is isomorphic to $C_3 \times C_3$.
\end{enumerate}
\end{thm}

\begin{thm}\label{dihedral}
Let $p$ be a prime number and $D_{2p}$ denotes the Dihedral group with $2p$ elements. Then the group $D_{2p}$ is Cay-DS if and only if $p\in \{2,3,5,7,11\}$.
In particular there exist two cospectral non-isomorphic $6$-regular Cayley graphs on the dihedral group of order $2p$  for any prime $p\geq 13$.
\end{thm}

 As we mentioned above it is shown in \cite{Ba} that $D_{2p}$ is not Cay-DS for any prime $p\geq 129$, so what Theorem \ref{dihedral} may have as a new result is about the primes less than $128$ and also the degree $6$ of the regularity  of non-DS Cayley graphs on  $D_{2p}$.

\begin{thm}\label{thm:minimal}
Every finite DS group is solvable.
\end{thm}


\section {\bf Preliminaries}

In this section we state some facts and results which are needed in the next sections.


Proposition 5.1 of \cite{LSV} shows that every subgroup of a finite Cay-DS group is also Cay-DS. The proof of the following proposition is the same as 
Proposition 5.1 of \cite{LSV}, we give it here for the reader's convenience. 
\begin{prop}\label{Non-DS}
Let $G$ be a finite group and $H$ be a subgroup of $G$. If $H$ is not DS (not Cay-DS, respectively), then $G$ is not DS (not Cay-DS, respectively).
In particular, the classes of finite DS groups and Cay-DS groups are closed under taking subgroups.
\end{prop}
\begin{proof}
By hypothesis there exists a Cayley graph $Cay(H,S)$ which is not DS (Cay-DS, respectively). 
Therefore there exists a graph $\Gamma$ (a Cayley graph $Cay(K,S')$ for some group $K$ and a symmetric subset $S'$ of $K$, repectively) which is  cospectral  and non-isomorphic to $Cay(H,S)$.
On the other hand, $Cay(G,S)$ has $|G:H|$ connected components isomorphic to $Cay(H,S)$.  Thus the disjoint union of $|G:H|$ copies of $\Gamma$ ($Cay(K,S')$ respectively) is a non-isomorphic cospectral mate of $Cay(G,S)$ and so $G$ is not DS. Note that $Cay(K\times C_{|G:H|},S' \times 1)$ is isomorphic to the disjoint union of $|G:H|$ copies of $Cay(K,S')$, where $S'\times 1=\{(x,1) \;|\; x\in S'\}$. Therefore, in the case $G$ is not Cay-DS, $H$ is not also Cay-DS. This completes the proof. 
\end{proof}

\begin{rem}
{\rm If $\Gamma$ and $\Gamma'$ are two non-isomorphic copsectral regular graphs, then their complements $\overline{\Gamma}$ and $\overline{\Gamma'}$ are also non-isomorphic copsectral. The complement of a Cayley graph $Cay(G,S)$ is  $Cay(G,\big(G\setminus S\big) \setminus \{1\})$. If $G\not=\langle S\rangle$ or equivalently $Cay(G,S)$ is disconnected, then $G= \langle  \big(G\setminus S\big) \setminus \{1\}\rangle$ or equivalently $Cay(G,\big(G\setminus S\big) \setminus \{1\})$ is connected. These observations show that a group is DS (Cay-DS, respectively) if and only if all of its {\em connected} Cayley graphs are DS (Cay-DS, respectively).} 
\end{rem}

We need the following  lemma to show that the property of being Cay-DS  is closed under taking quotient group.

\begin{lem}\label{jksimilar}
Let $A$ and $B$ be two $n\times n$ symmetric matrices and $J_k$ denotes the $k\times k$ matrix of all one. If for a permutation matrix $P_\sigma$ we have $P_{\sigma}\left(J_k\otimes A\right) P_{\sigma}^{-1}=J_k\otimes B$, then there is a permutation matrix $P_{\overline{\sigma}}$ such that $P_{\overline{\sigma}}AP_{\overline{\sigma}}^{-1}=B$. 
\end{lem}
\begin{proof}
It suffices to compute  indexes of permuted rows and columns of $J_k\otimes A$ under $P_{\sigma}$ and $P_{\sigma}^{-1}$ module  $n$. In this way, one can construct a permutation matrix $P_{\overline{\sigma}}$ of size $n$ such that $P_{\overline{\sigma}}AP_{\overline{\sigma}}^{-1}=B$ and this completes the proof.
\end{proof}

\begin{thm}\label{Qu}
Let $G$ be a finite group. Then $G$ is Cay-DS if and only if all of its quotient groups are Cay-DS.
\end{thm}
\begin{proof}
Let $G$ be a finite group  such that all of its quotient groups are Cay-DS. Then the quotient group $G/\{1\}$ is Cay-DS, which implies that the group $G$ is Cay-DS. Now suppose $G$ is Cay-DS and $N$ is a normal subgroup of $G$. We will show that the quotient group $G/N$ is also Cay-DS. We use a similar argument as in the proof of Lemma 3.5 of \cite{Azhvan}. Suppose that $G/N$ is not Cay-DS. Then there exist symmetric subsets $\overline{S}$ and $T$ of the quotient group $G/N$ and a group $H$, respectively such that two Cayley graphs $\overline{\Gamma}=Cay(G/N,\overline{S})$ and $Cay(H,T)$ are non-isomorphic and they are cospectral. 
Let $k=|N|$ and $e$ is the identity element of the cyclic group $C_k$ of order $k$. Note that $Cay(H,T)\cong \overline{\Gamma}'=Cay(\overline{H}, \overline{R})$, where $\overline{H}=H\times \{e\}$ and $\overline{R}=T\times \{e\}$. It follows that, there  exist  subsets $S_1$ and $R_1$ of  $G$ and $H\times C_{k}$ respectively,  such that we have $\overline{S}=\{Ns | s\in S_1\}$ and $\overline{R}=\{\overline{H} r | r\in R_1\}$. Let $S=\cup_{s\in S_1}Ns-\{1\}$ and $R=\cup_{r\in R_1}\overline{H}r-\{1\}$. It is easy to see that both $S$ and $R$ are symmetric subsets of  $G$ and $H\times C_{k}$ respectively and these subsets does not have the identity element of the corresponding groups. Let $A$ and $B$ be the adjacency matrix of the $\overline{\Gamma}$ and $\overline{\Gamma}'$ respectively. It is straightforward to see that the adjacency matrix of $Cay(G,S)$ and $Cay(H\times C_{k}, R)$ are $J_k\otimes A$ and $J_k\otimes B$ respectively, where  $J_k$ is the $k\times k$ matrix of all ones. The spectrum of these two latter adjacency matrices are equal, since $Spec(A)=Spec(B)$ and the eigenvalues of $J_k$ are $0$ with multiplicity $k-1$ and $k$ with multiplicity $1$. Now by Lemma \ref{jksimilar}, the two adjacency matrix $J_k\otimes A$ and $J_k\otimes B$ are not similar. So, we conclude that the group $G$ is not Cay-DS which is a contradiction. This completes our proof.
\end{proof}

The graph $Cay(G,S)$ is called Cayley isomorphism (for a short CI-graph) whenever if $Cay(G,S_{1})\cong Cay(G,S_{2})$, then there exists an element
$\sigma \in Aut(G)$ such that $S_{1}^\sigma=S_{2}$. A finite group $G$ is called a CI-group if all Cayley graphs of $G$ are CI-graphs.
All circulant CI-graphs (i.e. Cayley graphs over cyclic groups) are determined as follows.
\begin{thm}[\cite{Muz1, Muz2}]
The cyclic groups which are CI-groups are precisely those of order $n$, where $n$ is 8, 9, 18 or $n=2^{e}m$, where $e\in \{0,1,2\}$ and $m$ is odd and square-free.
\end{thm}
Dihedral groups $D_{2p}$ are also CI-groups.
\begin{thm}[\cite{B0}] \label{D2p} For any prime number $p$, the dihedral group $D_{2p}$ is a CI-group. 
\end{thm}
The Paley graph is a graph over a finite field $GF(q)$ such that $q\equiv 1$ (mod 4) and it is denoted by $P(q)$. The elements of the field are the vertices of Paley graph and two vertices are adjacent whenever their difference is a non-zero square in that field. it is well-known that Paley graph is a strongly regular graph with parameters $(v,k,\lambda,\mu)=(4t+1,2t,t-1,t)$ and it is a Cayley graph over the group $C_p^{k}$, where $q=p^{k}$. A strongly regular graph with these parameters is called conference graph. It is famous that a strongly regular graph is determined by its spectrum if and only if it is determined by its parameters.
The Paley graphs  $P(5)$, $P(13)$ and $P(17)$ are the unique strongly graphs with these parameters; on the other hand it can be found in \cite{BM} that there exist at least 41 strongly regular graphs with the same parameters as $P(29)$, at least 82 strongly regular graphs cospectral with $P(37)$ and at least 120 strongly regular graphs cospectral with $P(41)$. Furthermore, there are $15$ non-isomorphic strongly regular graphs with parameters $(25,12,5,6)$ (see \cite{PR}).

The Peisert graph, denoted by $\mathcal{P}^{*}(q)$ is defined for $q=p^{r}$, where $p$ is prime and $p\equiv 3$ (mod $4$) and $r$ is even. Let $a$ be a generator of $GF(q)$ and $M$ be a subset of $GF(q)$ as follows:
\begin{equation*}
M=\{a^{j}|j\equiv 0,1 (mod \ 4)\}.
\end{equation*}

The vertices of $\mathcal{P}^{*}(q)$ are the elements of $GF(q)$ and two vertices are adjacent if their difference is in $M$. Peisert graphs was first defined by Peisert in \cite{P}. The Paley graph $P(q)$ and the Peisert graph $\mathcal{P}^{*}(q)$ are strongly regular graphs with the same parameters.

\begin{lem} [Lemma 6.2, 6.4 and 6.6 of \cite{P}] \label{Paley}
If $p\equiv 3$ (mod $4$), then the $\mathcal{P}^{*}(p^{r})$ is not isomorphic to the Paley graph $P(p^{r})$, except when $p^{r}=3^{2}$.
\end{lem}
\begin{defn}
An affine plane is a point-line incidence structure which satisfies the following conditions: 

(i) Given any pair of points, there is exactly one line incident to both points.

(ii) Given a point $p$ and a line $l$ not incident to $p$, there is exactly one line $l'$ through $p$ which does not meet $l$.

(iii) There exists a set of four points, no three of them are incident to a common line.
\end{defn}

Let $AG(2, q)$ denote the affine plane of odd prime order $q$ derived from the ordered pairs of $\mathbb{F}_{q} \times \mathbb{F}_{q}$. Let $A$ denote a subset of
\begin{equation*}
(\bigcup_{y\in \mathbb{F}_{q}}l_{y}) \cup l_{\infty}.
\end{equation*}
Where,
\begin{equation*}
l_{y} = \{(c, cy) : c \in \mathbb{F}_{q}\},
l_{\infty} = \{(0, c) : c \in \mathbb{F}_{q}\}.
\end{equation*}
We define $G(A, q)$ to be a graph on the points of $AG(2, q)$ such that two points are adjacent if and only if the line incident to both points has slope $m$ such that $m\in A$. It is well-known that if $|A|=\frac{q+1}{2}$, then $G(A, q)$ is a conference graph. It is easy to see that all graphs of these types can consider as Cayley graphs over $C_{p}\times C_{p}$.

\vspace{1cc}
\section {\bf Forbidden subgroups of DS groups}
In view of Proposition \ref{Non-DS}, it is important to know which groups of prime order are not Cay-DS.
\begin{prop}\label{prop29}
The groups of orders $29$, $37$ and $41$ are not DS. 
\end{prop} 
\begin{proof}
It is known that there are exactly $41$ strongly $14$-regular graphs on $29$ vertices with parameters $(29,14,6,7)$ (see \cite[page 856]{Handbook}). One of these strongly regular graphs is the Paley graph $\mathcal{P}(29)$ which is a Cayley graph on the cyclic group of order $29$. It follows that there are $40$ non-isomorphic cospectral mates for $\mathcal{P}(29)$. This implies that $C_{29}$ is not DS. The same argument can be done for $37$ and $41$, noting that  there exist  at least $82$ strongly regular graphs cospectral with $P(37)$ and at least $120$ strongly regular graphs cospectral with $P(41)$ (see \cite{BM}).  
\end{proof}

In this section we prove that any finite group $G$ with a subgroup isomorphic to $C_{p}\times C_{p}$ for prime number $p\geq 5$ is non-DS. Also, we introduce some forbidden subgroups for DS groups. The Peisert graph $\mathcal{P}^{*}(p^{2})$ is not isomorphic to the Paley graph $P(p^{2})$, where $p=4t+3>7$ is prime by Lemma \ref{Paley} and in general case, in the following theorem, we prove that there exist cospectral non-isomorphic Cayley graphs on the group $C_{p}\times C_{p}$, where $p> 5$ is an arbitrary prime number. Let $GL(2,p)$ be the group of all invertible $2 \times 2$ matrices over a finite field $GF(p)$ and $PGL(2,p)$ be the quotient group of $GL(2,p)$ over the center of $GL(2,p)$, i.e. $\{cI_{2}|C_{p} \setminus \{0\} \}$ which is denoted by $Z(GL(2,p))$.\\

\begin{thm}\label{main1}
Let $G$ be a finite group and $p> 5$ be a prime number. If $H$ is a subgroup of $G$ which is isomorphic to $C_{p}\times C_{p}$, then $G$ is not Cay-DS.
If $C_5 \times C_5$ is  isomorphic to a subgroup of $G$, then $G$ is not DS. 
\end{thm}
\begin{proof}
As we mentioned before, the graph $G(A, p)$ is a conference graph whenever $|A|=\frac{p+1}{2}$. It is trivial to see that this graph is $Cay(C_{p}\times C_{p},S_{A})$, where $S_{A}=\{l_{y}|y \in A\} \setminus \{(0,0)\}$. Let $T$ be collections of all $S_{A}$. Following \cite{Ca}, $PGL(2,p)$ acts on the set $T$ (see section 4 of \cite{Ca}). Suppose in contrary that all Cayley graphs of these types are isomorphic. Therefore $T=\{(S_{A})^{g}|g\in PGL(2,p)\}$ for some fixed $A$ since $C_{p}\times C_{p}$ is a CI-group. Hence the order of an orbit of $S_{A}$, for some fixed $A$ is $[PGL(2,p):(PGL(2,p))_{S_{A}}]$. It follows that $|T|=[PGL(2,p):(PGL(2,p))_{S_{A}}]$. Thus
\begin{equation*}
|T|=\binom {p+1} {\frac{p+1}{2}}=[PGL(2,p):(PGL(2,p))_{S_{A}}]=\frac{|PGL(2,p)|}{|(PGL(2,p))_{S_{A}}|}.
\end{equation*}
Therefore $\binom {p+1} {\frac{p+1}{2}}$ divides $|PGL(2,p)|$ . On the other hand, $|PGL(2,p)|=\frac{(p^{2}-1)(p^{2}-p)}{p-1}$ and clearly $(p+1)!|(\frac{p+1}{2})!(\frac{p+1}{2})!(p-1)(p+1)p$. It follows that $(p-2)!|(\frac{p+1}{2})!(\frac{p+1}{2})!$ and this is a contradiction for $p\geq 7$. Furthermore, there are 15 non-isomorphic strongly regular graphs with parameters $(25,12,5,6)$ (see \cite{PR}) and this completes the proof by Proposition \ref{Non-DS}.
\end{proof}

\begin{prop}\label{pDS1}
The following groups are forbidden subgroups for a DS group:

\noindent 1) $C_{m}^{m}$, $m\geq 3$,

\noindent 2) $C_{4}\times C_{4}$,

\noindent 3) $C_{3}\times C_{6}$.
\end{prop}

\begin{proof}
We first note that the line graph of the complete bipartite graph $K_{m,n}$ is Cartesian product of two complete graphs $K_{m}$ and $K_{n}$. On the other hand, $K_{m,n}$ is not DS if and only if $\{m,n\} = \{4, 4\}$, $\{m, n\} = \{6, 3\}$, or $\{m, n\} = \{2t^{2} + t, 2t^{2} - t\}$ and there exists a strongly regular graph with spectrum $\{[2t^{2}]^{1}, [t ]^{2t^{2}-t-1}, [-t ]^{2t^{2}+t-1}\}$ (such a strongly regular graph comes from a symmetric Hadamard matrix with constant diagonal of size $4t^{2}$) (see \cite{VH}). On the other hand, it is easy to construct a Cayley graph over $C_{m}\times C_{n}$ which is Cartesian product of two complete graphs $K_{m}$ and $K_{n}$. Therefore 2, 3 are obvious. Recall that Hamming graph $H(m,n)$, is Cartesian product of $m$ complete graphs $K_{n}$ and $H(m,m)$ for $m\geq 3$ are not determined by the spectrum (see \cite{BH}) and this implies 1 and completes the proof.
\end{proof}

\begin{rem} {\rm Some parts of Proposition \ref{pDS1} can also be proved by the results the next sections. For example $C_4 \times C_4$ is not Cay-DS (see Table \ref{T01} and Theorem \ref{2groups}). 
}
\end{rem}

\section{\bf Cay-DS and DS $p$-Groups}
Using Theorem \ref{main1} we can prove Theorem \ref{corp}.\\

{\bf Proof of Theorem \ref{corp}.} 
Let $P$ be a Sylow $p$-subgroup ($5$-subgroup, respectively) of a finite Cay-DS group (DS group, respectively), where $p>3$. It follows from Theorem \ref{main1} and Proposition \ref{Non-DS} that $P$ has no subgroup isomorphic to $C_p \times C_p$. Since $p>2$, Lemma 1.4 of \cite{Ber} implies that $P$ is cyclic. $\hfill \Box$

It is known that all groups of order at most $9$ are DS (see \cite{VH}). In the following result, we prove that all groups of order at most $11$ are DS. 
\begin{prop}\label{at11}
Every group of order at most $11$ is DS. 
\end{prop}
\begin{proof}
There are two non-isomorphic groups of order $10$, specially the cyclic group  $C_{10}$ and the dihedral group $D_{10}$. The total number of non-isomorphic connected Cayley graphs over $D_{10}$ and $C_{10}$ are $16$ and the non-isomorphic mates are not cospectral. Therefore, the latter groups are Cay-DS. By  [\cite{CD}, p. 398], there are only two pairs $(\Gamma_1,\Gamma_2)$ and $(\Gamma_1^{c},\Gamma_2^{c})$  of regular non-isomorphic cospectral graphs with $10$ vertices (here $\Gamma^{c}$ denotes the complement of the graph $\Gamma$). The graphs $\Gamma_1$ and $\Gamma_2$  are shown in the Figure \ref{CRP1}. We have checked that non of these four graphs are cospectral with any of $32$ Cayley graphs over the groups of order $10$. Hence all group of order at most $10$ are DS. 

The only group of order $11$ is $C_{11}$ and by the main result of \cite{DJO}, $C_{11}$ is Cay-DS. Now we prove that $C_{11}$ is also DS. The total number of non-isomorphic connected Cayley graphs on $C_{11}$ is equal to $7$, as follows:  one is $2$-regular, two are $4$-regular, two are $6$-regular, one is $8$-regular and one is $10$-regular. It is easy to see that the $8$-regular graph is the complement of the $2$-regular graph and so, both of them are DS. Also, the $10$-regular graph is DS, since it is the complete graph with $11$ vertices. The two remaining $6$-regular graphs are the complement of two $4$-regular graphs. The characteristic polynomials of these two $4$-regular graphs are $(x-4)( x^5+2x^4-5x^3-13x^2-7x-1)^2$ and $(x-4) (x^5+2x^4-5x^3-2x^2+4x-1)^2$. We have $265$ non-isomorphic $4$-regular graphs and none of them are  cospectral mate of our two $4$-regular Cayley graphs. This completes our proof.
\end{proof}

\begin{figure}[htb]
\centering
\begin{tikzpicture}[scale=0.85]
\node (1) [circle, minimum size=4 pt, fill=black, line width=0.250 pt, draw=black] at (100.0pt, -75.0pt)  {};
\node (2) [circle, minimum size=4 pt, fill=black, line width=0.250 pt, draw=black] at (150.0pt, -75.0pt)  {};
\node (3) [circle, minimum size=4 pt, fill=black, line width=0.250 pt, draw=black] at (200.0pt, -75.0pt)  {};
\node (4) [circle, minimum size=4 pt, fill=black, line width=0.250 pt, draw=black] at (112.5pt, -100.0pt)  {};
\node (5) [circle, minimum size=4 pt, fill=black, line width=0.250 pt, draw=black] at (187.5pt, -100.0pt)  {};
\node (6) [circle, minimum size=4 pt, fill=black, line width=0.250 pt, draw=black] at (112.5pt, -137.5pt)  {};
\node (7) [circle, minimum size=4 pt, fill=black, line width=0.250 pt, draw=black] at (187.5pt, -137.5pt)  {};
\node (8) [circle, minimum size=4 pt, fill=black, line width=0.250 pt, draw=black] at (100.0pt, -162.5pt)  {};
\node (9) [circle, minimum size=4 pt, fill=black, line width=0.250 pt, draw=black] at (150.0pt, -162.5pt)  {};
\node (10) [circle, minimum size=4 pt, fill=black, line width=0.250 pt, draw=black] at (200.0pt, -162.5pt)  {};
\node (11) [circle, minimum size=4 pt, fill=black, line width=0.250 pt, draw=black] at (262.5pt, -75.0pt)  {};
\node (12) [circle, minimum size=4 pt, fill=black, line width=0.250 pt, draw=black] at (312.5pt, -75.0pt)  {};
\node (13) [circle, minimum size=4 pt, fill=black, line width=0.250 pt, draw=black] at (362.5pt, -75.0pt)  {};
\node (14) [circle, minimum size=4 pt, fill=black, line width=0.250 pt, draw=black] at (275.0pt, -100.0pt)  {};
\node (15) [circle, minimum size=4 pt, fill=black, line width=0.250 pt, draw=black] at (350.0pt, -100.0pt)  {};
\node (16) [circle, minimum size=4 pt, fill=black, line width=0.250 pt, draw=black] at (275.0pt, -137.5pt)  {};
\node (17) [circle, minimum size=4 pt, fill=black, line width=0.250 pt, draw=black] at (350.0pt, -137.5pt)  {};
\node (18) [circle, minimum size=4 pt, fill=black, line width=0.250 pt, draw=black] at (262.5pt, -162.5pt)  {};
\node (19) [circle, minimum size=4 pt, fill=black, line width=0.250 pt, draw=black] at (362.5pt, -162.5pt)  {};
\node (20) [circle, minimum size=4 pt, fill=black, line width=0.250 pt, draw=black] at (312.5pt, -162.5pt)  {};
\draw [line width=0.75, color=black] (1) to  (2);
\draw [line width=0.75, color=black] (2) to  (3);
\draw [line width=0.75, color=black] (3) to  (10);
\draw [line width=0.75, color=black] (1) to  (8);
\draw [line width=0.75, color=black] (8) to  (9);
\draw [line width=0.75, color=black] (10) to  (9);
\draw [line width=0.75, color=black] (3) to  (5);
\draw [line width=0.75, color=black] (2) to  (5);
\draw [line width=0.75, color=black] (1) to  (4);
\draw [line width=0.75, color=black] (2) to  (4);
\draw [line width=0.75, color=black] (3) to  (7);
\draw [line width=0.75, color=black] (1) to  (6);
\draw [line width=0.75, color=black] (5) to  (7);
\draw [line width=0.75, color=black] (4) to  (6);
\draw [line width=0.75, color=black] (7) to  (9);
\draw [line width=0.75, color=black] (6) to  (9);
\draw [line width=0.75, color=black] (7) to  (8);
\draw [line width=0.75, color=black] (6) to  (10);
\draw [line width=0.75, color=black] (8) to  (4);
\draw [line width=0.75, color=black] (10) to  (5);
\draw [line width=0.75, color=black] (11) to  (12);
\draw [line width=0.75, color=black] (12) to  (13);
\draw [line width=0.75, color=black] (13) to  (19);
\draw [line width=0.75, color=black] (20) to  (19);
\draw [line width=0.75, color=black] (18) to  (20);
\draw [line width=0.75, color=black] (11) to  (18);
\draw [line width=0.75, color=black] (13) to  (15);
\draw [line width=0.75, color=black] (13) to  (16);
\draw [line width=0.75, color=black] (11) to  (14);
\draw [line width=0.75, color=black] (11) to  (17);
\draw [line width=0.75, color=black] (19) to  (17);
\draw [line width=0.75, color=black] (19) to  (15);
\draw [line width=0.75, color=black] (18) to  (16);
\draw [line width=0.75, color=black] (18) to  (14);
\draw [line width=0.75, color=black] (20) to  (17);
\draw [line width=0.75, color=black] (20) to  (16);
\draw [line width=0.75, color=black] (14) to  (12);
\draw [line width=0.75, color=black] (12) to  (15);
\draw [line width=0.75, color=black] (15) to  (17);
\draw [line width=0.75, color=black] (14) to  (16);
\end{tikzpicture}
\caption{The regular cospectral graphs with cospectral complement }\label{CRP1}
\end{figure}
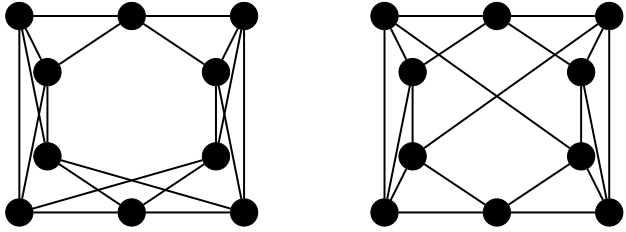

In the following, we use the symbol $S_G$ to show that $S$ is a symmetric subset of the group $G$. Also, the symbol $\left(S_{G_1},S_{G_2}\right)$ means two Cayley graphs $Cay(G_1,S_{G_1})$ and $Cay(G_2,S_{G_2})$ are non-isomorphic cospectral mates. We have used the software {\sf GAP} \cite{gap} to do our calculations and writing algorithms. Also, when we write $G_{i}^{n}$, it means that the group $G$ is of order $n$ and has $i$-th position in the standard library of group theory package in {\sf GAP}.

We need the following lemma in the proof of Theorem \ref{2groups}.
\begin{lem}\label{SpC32}
The only group of order $32$  in which all of whose  maximal subgroups are isomorphic to either $G_{12}^{16}$ or $G_{13}^{16}$ is isomorphic to $G_{50}^{32}$.
\end{lem}
\begin{proof}
It can be checked by {\sf GAP} \cite{gap} using its Small Groups Library. 
\end{proof}

In the following, we characterize all finite Cay-DS $2$-groups. Indeed all  finite Cay-DS $2$-groups are DS with two possible exceptions  $G_{12}^{16}$ and $G_{13}^{16}$.

\begin{thm}\label{2groups}
Let $G$ be a $2$-group of order $2^n$, $n\geq 1$. Then

\noindent 1) if $1\leq n\leq 3$, then $G$ is DS,

\noindent 2) if $n=4$, then  $G$ is Cay-DS if and only if $G\cong G_{12}^{16}$ or $G_{13}^{16}$.

\noindent 3) if $n\geq 5$, then $G$ is not Cay-DS.
\end{thm}

\begin{proof}
1) \; The proof follows from Proposition \ref{at11}.\\

2) \; We have to notice that there are exactly $14$ non-isomorphic groups of order $16$ (e.g., use {\sf AllSmallGroups(16)} of {\sf GAP} \cite{gap}). The presentations of  groups of order $16$ are given in Table \ref{T01}. In Table \ref{T11}, we exhibit  non-isomorphic cospectral Cayley graphs for $12$ of these groups with the exceptions  $G_{12}^{16}$ and $G_{13}^{16}$. 
This means that all groups of order $16$ are not Cay-DS with possible exceptions $G_{12}^{16}$ and $G_{13}^{16}$. 
To prove that the latter groups are Cay-DS, we have simply  constructed all  connected Cayley graphs over  all  groups of order $16$ and examined that each connected Cayley graph over $G_{12}^{16}$ or $G_{13}^{16}$ is not   copsectral with a non-isomorphic  Cayley graphs over any group of order $16$. In this process, we have used {\sf GRAPE}  package \cite{grape} of {\sf GAP} \cite{gap}. \\

 3) \; If $n=5$, then by Lemma \ref{SpC32} and Proposition \ref{Non-DS}, all groups of order $32$, except possibly $G_{50}^{32}$, are not Cay-DS. It remains to show that $G_{50}^{32}$ is not Cay-DS. The latter is done in Table \ref{T21}, by exhibiting a Cayley graph over $G_6^{32}$ which is copspectral and non-isomorphic to a Cayley graph over $G_{50}^{32}$.  Therefore all groups of order $32$ are not Cay-DS. \\
 Now since each $2$-group of order $2^n$ ($n\geq 5$) contains a subgroup of order $32$, it follows from Proposition \ref{Non-DS} that  each $2$-group of order at least $32$ is not Cay-DS.  This completes the proof. 
\end{proof}

{\noindent \bf Proof of Theorem \ref{corp} (2).} It follows from Theorem \ref{2groups}. \\
 
\begin{table}[h]
\centering
\tiny
\caption{The presentation of groups of order $16$ }\label{T01}
\begin{tabular}{|l|}\hline\\
$G_1^{16}=C_{16}=\left\langle a \;|\; a^{16}=1 \right\rangle$\\
$G_2^{16}=C_4\times C_4=\left\langle a,b|a^4=b^4=1,ab=ba\right\rangle$\\
$G_3^{16}=(C_4\times C_2):C_2=\left\langle a,b|a^4=b^2=(a^2b)^2=(ab)^4=1\right\rangle$\\
$G_4^{16}=C_4:C_4=\left\langle a,b|a^4=b^4=1, a^{-1}ba=b^{-1}\right\rangle$\\
$G_5^{16}=C_8\times C_2=\left\langle a,b|a^8=b^2=1, ab=ba\right\rangle$\\
$G_6^{16}=C_8:C_2=\left\langle a,b|a^8=b^2=1, bab=a^{5}\right\rangle$\\
$G_7^{16}=D_{16}=\left\langle a,b|a^2=b^2=1, (ab)^8=1\right\rangle$\\
$G_8^{16}=QD_{16}=\left\langle a,b|a^4=b^2=1, (a^2b)^2=1, (ab)^3=ba\right\rangle$\\
$G_9^{16}=Q_{16}=\left\langle a,b|a^4=b^4=1, a^2=b^2, (ab)^3=ba^{-1}\right\rangle$\\
$G_{10}^{16}=C_4\times C_2\times C_2=\left\langle a,b,c|a^4=b^2=c^2=1, ab=ba, ac=ca, bc=cb\right\rangle$\\
$G_{11}^{16}=C_2\times D_8=\left\langle a,b,c|a^2=b^2=c^2=1, (ac)^2=1, (bc)^2=1, (ab)^4=1\right\rangle$\\
$G_{12}^{16}=C_2\times Q_8=\left\langle a,b,c|c^2=1, a^2=b^2=b^{-2}, a^{-1}ba=b^{-1}, ac=ca, bc=cb\right\rangle$\\
$G_{13}^{16}=(C_4\times C_2):C_2=\left\langle a,b,c|a^2=b^2=c^4=1, ac=ca, bc=cb, c^2=(ba)^2\right\rangle$\\
$G_{14}^{16}=C_2\times C_2\times C_2\times C_2=\left\langle a,b,c,d|a^2=b^2=c^2=d^2=(ab)^2=(ac)^2=(bc)^2=(ad)^2=(bd)^2=(cd)^2=1\right\rangle$\\
\\\hline
\end{tabular}
\end{table}

Note that in Table \ref{T01} the group $G_2^{16}$ and $G_{13}^{16}$ have the same  ``structure description"  $(C_4\times C_2):C_2$ according to the function 
{\sf StructureDescription} of {\sf GAP} \cite{gap} but of course they are non-isomorphic. 

\begin{table}[h]
\centering
\tiny
\caption{Non-DS groups of order $16$ with their non-isomorphic cospectral mates}\label{T11}
\begin{tabular}{|l|}\hline
$\quad\quad\quad\quad\quad\quad\quad\quad\quad\quad\quad\quad\quad\quad\quad\quad \left(S_{G_1},S_{G_2}\right)$\\\hline\\
$\left(\left\{a^{\pm 1}, b, a^2, aba^{-1}b\right\}_{G_3^{16}},\left\{a, b, c, bc, c(ab)^2\right\}_{G_{11}^{16}}\right)$\\
$\left(\left\{a^{\pm 1}, b^{\pm 1}, b^2, (ab)^2\right\}_{G_2^{16}},\left\{a^{\pm 1}, b^{\pm 1}, a^2, b^2\right\}_{G_4^{16}}\right)$\\
$\left(\left\{a^{\pm 1}, b, a^4, ab,aba^6\right\}_{G_5^{16}},\left\{a, b, (ab)^4, a(ab)^6, a(ab)^4, b(ab)^4\right\}_{G_7^{16}}\right)$\\
$\left(\left\{a^{\pm 1}, b, a^4, aba^4,aba^6\right\}_{G_6^{16}},\left\{a^{\pm 1}, b, a^2, aba^2, ab(a^3b)^2\right\}_{G_8^{16}}\right)$\\
$\left(\left\{a^{\pm 1}, b^{\pm 1}, b^2, (ab)^2\right\}_{G_2^{16}},\left\{a^{\pm 1}, b, c, a^2, bc\right\}_{G_{10}^{16}}\right)$\\
$\left(\left\{a^{\pm 1}, b^{\pm 1}, b^2, (ab)^2\right\}_{G_2^{16}},\left\{a, b, c, d, ab, ad\right\}_{G_{14}^{16}}\right)$\\
$\left(\left\{a^{\pm 1}, a^{\pm 4}, a^{\pm 5}, a^8\right\}_{G_1^{16}},\left\{a^{\pm 1}, b, aba^{-1}b, baba^{-1}b, ba^2, aba^{-1}ba^2\right\}_{G_3^{16}}\right)$\\
$\left(\left\{a, b^{\pm 1}, a^3ba^3b^3, b^2, ab^2, (a^3b)^2\right\}_{G_9^{16}},\left\{a^{\pm 1}, b, aba^{-1}b, baba^{-1}b, ba^2, aba^{-1}ba^2\right\}_{G_3^{16}}\right)$\\
\\\hline
\end{tabular}
\end{table}

\begin{table}[h]
\centering
\tiny
\caption{The group $(C_2\times Q_8):C_2$ is not Cay-DS}\label{T21}
\begin{tabular}{|l|}\hline\\
$G_6^{32}=\left\langle a,b|a^4=b^2=1, (ba^{-1})^4=1, (ba^{-1}ba)^2=1, a(ba^{-2})^3ba=1\right\rangle$\\
$G_{50}^{32}=\left\langle a,b,c,d|a^2=d^2=b^4=c^4=1, b^2c^2=1, bc^{-1}b^{-1}c^{-1}=1, (ab^{-1})^2=1, acac^{-1}=1, bdb^{-1}d=1, cdc^{-1}d=1, ac^2dad=1\right\rangle$\\
$\quad\quad\quad\quad\quad\quad\quad\quad\quad\quad\quad\quad\quad\left(\left\{a^{\pm 1}, b, (ba)^2(ab)^2, ba(a^2b)^3ab\right\}_{G_6^{32}},\left\{a, b, d, bc^2, acd\right\}_{G_{50}^{32}}\right)$\\
\\\hline
\end{tabular}
\end{table}

In the following, we study Cay-DS $3$-groups. All groups of order $9$, which are isomorphic to $C_9$ or $C_3\times C_3$, are Cay-DS by Proposition \ref{at11}. 

\begin{thm}\label{3gr}
Let $G$ be a group of order $27$. Then $G$ is Cay-DS if and only if $G\cong C_{27}$.
\end{thm}
\begin{proof}
There are exactly $5$ non-isomorphic groups of order $27$. In Table \ref{T31},  presentations of all these groups except $C_{27}$ are given. Also, a non-isomorphic cospectral Cayley graphs for  each of the latter $4$ groups is exhibited in Table \ref{T41}. The only remaining group is $C_{27}$ for which there are $920$ non-isomorphic connected Cayley graphs. None of the latter Cayley graphs has a non-isomorphic cospectral mate among Cayley graphs over groups of order $27$. This claim is checked by the {\sf GARPE} package of {\sf GAP}. This completes the proof.
\end{proof}

{\noindent Proof of Theorem \ref{corp} (3).} 
Let $P$ be a Sylow $3$-subgroup of a finite Cay-DS groups which is not isomorphic to $C_3 \times C_3$. If $|P|<27$, then by assumption, $P$ is cyclic of order $1$, $3$ or $9$. So we may assume that $|P|\geq 27$.
Then $P$ contains a subgroup  of order $27$. Let $Q$ be any subgroup of order $27$ of $P$.
It follows from Theorem \ref{3gr} and Proposition \ref{Non-DS} that $Q$ is cyclic of order $27$. 
Since every subgroup of $P$ of order at most $27$ is contained in a subgroup of order $27$, it follows that $P$ has no subgroup isomorphic to $C_3 \times C_3$.
Now Lemma 1.4 of \cite{Ber} implies that $P$ is cyclic. $\hfill \Box$\\

\begin{table}[h]
\centering
\tiny
\caption{The presentation of groups of order $27$ }\label{T31}
\begin{tabular}{|l|}\hline\\
$G_2^{27}=C_9\times C_3=\left\langle a,b| a^9=b^3=1, ab=ba\right\rangle$\\
$G_3^{27}=(C_3\times C_3):C_3=\left\langle a,b| a^3=b^3= (ab)^3= (ab^{-1})^3=1\right\rangle$\\
$G_4^{27}=C_9:C_3=\left\langle a,b| a^9=b^3=1, b^{-1}ab=a^{-2}\right\rangle$\\
$G_5^{27}=C_3\times C_3\times C_3=\left\langle a,b,c| a^3=b^3=c^3=1, ab=ba, ac=ca, bc=cb\right\rangle$\\
\\\hline
\end{tabular}
\end{table}

\begin{table}[h]
\centering
\tiny
\caption{Non-DS groups of order $27$ with their non-isomorphic cospectral mates}\label{T41}
\begin{tabular}{|l|}\hline
$\quad\quad\quad\quad\quad\quad\quad\quad\quad\quad\quad\quad\quad\quad\quad\quad \left(S_{G_1},S_{G_2}\right)$\\\hline\\
$\left(\left\{a^{\pm 1}, b^{\pm 1}, abab^{-1}a^{-1}, a^2(bab^{-1}a^{-1})^2\right\}_{G_3^{27}},\left\{a^{\pm 1}, b^{\pm 1}, ab^2, a^2ba^3\right\}_{G_4^{27}}\right)$\\
$\left(\left\{a^{\pm 1}, b^{\pm 1}, b^2ab^{-1}a^{-1}, b^2(bab^{-1}a^{-1})^2\right\}_{G_3^{27}},\left\{a, b, c, a^2, b^2, c^2\right\}_{G_5^{27}}\right)$\\
$\left(\left\{a^{\pm 1}, b^{\pm 1}, a^{\pm 3}, ab^2, a^2ba^6, b^2a^3, ba^6\right\}_{G_2^{27}},\left\{a^{\pm 1}, b^{\pm 1}, a^{\pm 3}, ba^3, a^2b^2, aba^6, b^2a^6\right\}_{G_2^{27}}\right)$\\
\\\hline
\end{tabular}
\end{table}

\section{\bf Cay-DS dihedral Groups}
In this section we prove Theorem \ref{dihedral}, that is we determine exactly when the dihedral group $D_{2p}$ is Cay-DS for any prime number $p$. \\

{\noindent \bf Proof of Theorem \ref{dihedral}.}
By Proposition \ref{at11}, the dihedral group $D_{2p}$ for $p=2,3$ and $5$ is DS. Also, by {\sf GAP} \cite{gap} it is checked  that $D_{2p}$ is Cay-DS for  $p=7$ and $p=11$. Now for each prime number $p\geq 13$, we construct a pair of non-isomorphic cospectral Cayley graphs of degree $6$ over $D_{2p}$. Suppose $a$ and $b$ denote two arbitrary but from now on fixed elements of $D_{2p}$ of orders $2$ and $p$ respectively. Let $S=\{a,ab,ab^2,ab^6,ab^8,ab^{11}\}$ and $T=\{a,ab^2,ab^4,ab^5,ab^{10},ab^{11}\}$. Note that  all elements of subsets $S$ and $T$ are involutions. For a set $X$ of integers, let $\beta_X (c)$, $0 \leq c \leq p-1$, denotes the total number of solutions $(x,y)\in X\times X$ of equation $$x-y\equiv c \left(\text{mod} \; p\right).$$
It is easy to see that for each integer number $c$, $0\leq c \leq p-1$, $\beta_S (c)=\beta_T (c)$. It follows from  Corollary 4.2 of \cite{Ba}, two Cayley graphs $Cay\left(D_{2p},S\right)$ and $Cay\left(D_{2p},T\right)$ are cospectral. Now we show that these two Cayley graphs are not isomorphic. Since $D_{2p}$ is a CI-group \ref{D2p},  it is sufficient to show that there is no automorphism $\sigma \in Aut(D_{2p})$ such that $T^\sigma=S$. An arbitrary automorphism  of $D_{2p}$ has the form as follows 
\begin{equation*}
\sigma : \left\{
\begin{array}{lr}
a\mapsto ab^l &  0\leq l \leq p-1\\
b\mapsto b^s & 1\leq s \leq p-1
\end{array} \right.
\end{equation*}
Therefore, we must show that two sets $S$ and $T^\sigma =\{l,l+2s,l+4s,l+5s,l+10s,l+11s\}$ are not equal (in mod $p$), for any values of $0\leq l \leq p-1$ and $1\leq s \leq p-1$. Suppose, for a contradiction, that $S=T^\sigma$, then there are suitable $l$ and $s$ such that we have
$$\{l+a_1s=0,l+a_2s=1,l+a_3s=2,l+a_4s=6,l+a_5s=8,l+a_6s=11\},$$
where, $a_i\neq a_j, i\neq j$ and $a_i \in \{0,2,4,5,10,11\}, 1\leq i \leq 6.$

Since exactly one of $a_i$ is zero, there exist  at most six possible values for the parameter $l$ which are $0,1,2,6,8$ or $11$. Now we investigate each possible case for $l$ and show that in each case we get a contradiction. 
\begin{itemize}
\item[case 1:] $a_1=0$\\
Then $l=0$ and $a_5-a_4=a_3$, which  is impossible.
\item[case 2:] $a_2=0$\\
It follows that $l=1$ and $a_5-a_4=2a_3$, a contradiction.
\item[case 3:] $a_3=0$\\
Thus $l=2$ and $a_4-a_5=2a_2$, a contradiction.
\item[case 4:] $a_4=0$\\
Then $l=6$ and $a_2-a_1=a_3$, a contradiction.
\item[case 5:] $a_5=0$\\
So $l=8$ and $a_1-a_3=a_4$, which is impossible. 
\item[case 6:] $a_6=0$\\
This case is a little challenging. It follows that $l=11$ and $5a_5=3a_4$.  Thus we must determine  for which primes $p$ the latter equalities are valid. We have to examine all cases where $a_4\in \{2,4,5,10,11\}$ and $a_5 \in \{2,4,5,10,11\}$ such that $a_4\neq a_5$. Checking by hand and doing some simple calculations (noting that  $p\geq 13$) we find that if $p\notin \{13,17,19,23,43\}$, then there is no solution for the equation $5a_5=3a_4$.
\end{itemize}
Therefore, to complete the proof, we must only check the equality $S=T^\sigma$ for the  primes $p\in \{13,17,19,23,43\}$. It is done by simple calculations in {\sf GAP} \cite{gap} and the latter equality is not possible for the primes $p\in \{13,17,19,23,43\}$.  This completes the proof. $\hfill \Box$\\

\begin{cor}\label{cordihed}
Let $D_{2n}$ be the dihedral group of order $2n$. If $D_{2n}$ is Cay-DS, then $n\in \{1,2,3,4,5,7,9\}$ or $n=11^k$ for some integer $k\geq 1$. If $n\in\{1,2,3,4,5,7,9,11\}$ then $D_{2n}$ is Cay-DS.
\end{cor}
\begin{proof}
It follows from Proposition \ref{at11}  $D_{2n}$ is Cay-DS for every $n\in \{1,2,3,4,5\}$. Also by computation it is verified that for $n\in \{7,9,11\}$ the group $D_{2n}$ are Cay-DS. Now suppose that $D_{2n}$ is Cay-DS. By Theorem \ref{dihedral}, the factorization of $n$ must be of the form $n=2^m3^q5^r7^s11^k$, for some  integers $m,q,r,s$ and $k$. By the results in the Table \ref{T61}, we have $n=11^k$, which completes our proof.
\begin{table}[h]
\centering
\tiny
\caption{Some special dihedral groups which are not Cay-DS}\label{T61}
\begin{tabular}{|l|}\hline\\
$D_{12}=\left\langle a,b| a^2=b^6=1, (ab^{-1})^2=1\right\rangle$\\
$\quad\left(\left\{a,b^{\pm 1}, ab\right\},\left\{a, ab, b^3, ab^3\right\}\right)$\\
$D_{16}=\left\langle a,b| a^2=b^8=1, (ab^{-1})^2=1\right\rangle$\\
$\quad\left(\left\{a, b^{\pm 3}, ab, ab^2, ab^5\right\},\left\{ab, b^{\pm 3}, ab^2, ab^4, ab^5\right\}\right)$\\
$D_{20}=\left\langle a,b| a^2=b^{10}=1, (ab^{-1})^2=1\right\rangle$\\
$\quad\left(\left\{a, ab, ab^2, ab^3, b^5, ab^6\right\},\left\{a, ab, ab^2, ab^4, b^5, ab^5\right\}\right)$\\
$D_{28}=\left\langle a,b| a^2=b^{14}=1, (ab^{-1})^2=1\right\rangle$\\
$\quad\left(\left\{a, ab, ab^2, ab^3, ab^4, ab^8\right\},\left\{a, ab, ab^2, ab^3, ab^5, ab^7\right\}\right)$\\
$D_{30}=\left\langle a,b| a^2=b^{15}=1, (ab^{-1})^2=1\right\rangle$\\
$\quad\left(\left\{a, ab, ab^2, ab^3, ab^6, ab^7\right\},\left\{a, ab, ab^2, ab^3, ab^5, ab^8\right\}\right)$\\
$D_{42}=\left\langle a,b| a^2=b^{21}=1, (ab^{-1})^2=1\right\rangle$\\
$\quad\left(\left\{a, ab, ab^2, ab^3, ab^7, ab^{15}\right\},\left\{a, ab, ab^3, ab^4, ab^{10}, ab^{14}\right\}\right)$\\
$D_{44}=\left\langle a,b| a^2=b^{22}=1, (ab^{-1})^2=1\right\rangle$\\
$\quad\left(\left\{a, ab, ab^2, ab^3, ab^9, ab^{13}\right\},\left\{a, ab, ab^2, ab^4, ab^{10}, ab^{11}\right\}\right)$\\
$D_{50}=\left\langle a,b| a^2=b^{25}=1, (ab^{-1})^2=1\right\rangle$\\
$\quad\left(\left\{a, ab, ab^2, ab^5, ab^{11}, ab^{15}\right\},\left\{a, ab, ab^5, ab^6, ab^{10}, ab^{17}\right\}\right)$\\
$D_{54}=\left\langle a,b| a^2=b^{27}=1, (ab^{-1})^2=1\right\rangle$\\
$\quad\left(\left\{a, ab, ab^2, ab^3, ab^9, ab^{19}\right\},\left\{a, ab, ab^3, ab^9, ab^{10}, ab^{11}\right\}\right)$\\
$D_{66}=\left\langle a,b| a^2=b^{33}=1, (ab^{-1})^2=1\right\rangle$\\
$\quad\left(\left\{a, ab, ab^2, ab^3, ab^{12}, ab^{22}\right\},\left\{a, ab, ab^3, ab^{11}, ab^{23}, ab^{24}\right\}\right)$\\
$D_{70}=\left\langle a,b| a^2=b^{35}=1, (ab^{-1})^2=1\right\rangle$\\
$\quad\left(\left\{a, ab, ab^2, ab^5, ab^{16}, ab^{23}\right\},\left\{a, ab, ab^5, ab^8, ab^{14}, ab^{26}\right\}\right)$\\
$D_{98}=\left\langle a,b| a^2=b^{49}=1, (ab^{-1})^2=1\right\rangle$\\
$\quad\left(\left\{a, ab, ab^7, ab^9, ab^{11}, ab^{14}\right\},\left\{a, ab, ab^7, ab^{15}, ab^{33}, ab^{40}\right\}\right)$\\
$D_{110}=\left\langle a,b| a^2=b^{55}=1, (ab^{-1})^2=1\right\rangle$\\
$\quad\left(\left\{a, ab, ab^2, ab^5, ab^{26}, ab^{33}\right\},\left\{a, ab, ab^5, ab^{31}, ab^{32}, ab^{53}\right\}\right)$\\
$D_{154}=\left\langle a,b| a^2=b^{77}=1, (ab^{-1})^2=1\right\rangle$\\
$\quad\left(\left\{a, ab, ab^7, ab^9, ab^{11}, ab^{14}\right\},\left\{a, ab, ab^3, ab^8, ab^{10}, ab^{14}\right\}\right)$\\
\\\hline
\end{tabular}
\end{table}
\end{proof}

We do not know whether  $D_{2\cdot 11^2}=D_{242}$ is not Cay-DS. If it is the case, we completely determine positive integers $n$ for which  $D_{2n}$ is Cay-DS. 

\section{\bf Finite DS groups are solvable}

In the next lemma, we determine the Cay-DS property of some groups which are needed to prove the next theorem.

\begin{lem}\label{table}
The groups $\left( C_2\times C_2\times C_2\right):C_7=\left\langle a,b| a^7=b^2=(aba^{-1}b)^2= aba^{-3}(ba)^2=1\right\rangle$ and $A_5$ are not Cay-DS.
\end{lem}
\begin{proof}
In the Table \ref{T51}, we give the presentations of the above groups and two inverse close subsets of these groups which the Cayley graphs derived from them are not isomorphic but are cospectral. The computation done by {\sf GAP} \cite{gap}.
\begin{table}[h]
\centering
\tiny
\caption{Some special groups which are not Cay-DS}\label{T51}
\begin{tabular}{|l|}\hline\\
$G_{11}^{56}=\left\langle a,b| a^7=b^2=(aba^{-1}b)^2=aba^{-3}(ba)^2=1\right\rangle$\\
$\quad\left(\left\{b, a^{\pm 2}, aba^{-1}, a^3(ba^{-1})^2b^{-1}a, a^4b^{-1}a\right\}_{G_{11}^{56}},\left\{b, a^{\pm 3}, aba^{-1}, a^4ba^{-1}b^{-1}a, a^2b^{-1}a\right\}_{G_{11}^{56}}\right)$\\
$A_5$\\
$\quad(\left\{(2,3)(4,5), (2,4)(3,5), (2,5)(3,4), (1,2)(4,5), (1,2)(3,4), (1,3)(4,5), (1,4)(3,5)\right\},$\\
$\; \; \; \; \;\left\{(2,3)(4,5), (2,4)(3,5), (2,5)(3,4), (1,2)(4,5), (1,2)(3,4), (1,3)(4,5), (1,4)(2,3)\right\})$\\
\\\hline
\end{tabular}
\end{table}
\end{proof}

\noindent{\bf Proof of Theorem \ref{thm:minimal}.}
Suppose, for a contradiction, that $G$ is a finite non-solvable DS group of minimum order.  Since every subgroup and quotient of a DS group is DS by Proposition \ref{Non-DS} and Theorem \ref{Qu}) and the class of solvable groups is closed under extensions and subgroups, $G$ is a finite non-abelian simple group. 
By Theorem 1 of \cite{BW}, $G$ contains a minimal simple group $H$. By Thompson's classification of  minimal simple groups \cite{Thompson}, $H$ is isomorphic to one the following groups: 
\begin{itemize}
\item[1)] $\text{PSL}_2\left( 2^p\right)$, $p$ a prime,
\item[2)] $\text{PSL}_2\left( 3^p\right)$, $p$ an odd prime,
\item[3)] $\text{PSL}_2\left( p\right)$, $p> 3$ a prime congruent to $2$ or $3$ mod $5$,
\item[4)] $\text{Sz}\left( 2^p\right)$, $p$ an odd prime,
\item[5)] $\text{PSL}_3\left( 3\right)$.
\end{itemize}
Now we prove that none of the above groups are DS.
The group $\text{PSL}_3\left( 3\right)$ is not DS: for it can be easily seen by {\sf GAP} \cite{gap} that  $\text{PSL}_3(3)$ has a subgroup isomorphic to  $D_{12}$, which is not DS by Corollary \ref{cordihed}.\\

The group $\text{PSL}_2\left( p\right)$, $p> 3$ a prime congruent to $2$ or $3$ mod $5$, has the dihedral groups $D_{p-1}$ and $D_{p+1}$ as its subgroups \cite[HauptSatz 8.27, p. 213]{Hup1}. By Theorem \ref{dihedral}, we have $p-1=2^{t_1}3^{t_2}7^{t_3}11^{t_4}$ and $p+1=2^{s_1}3^{s_2}7^{s_3}11^{s_4}$, since $p$ is congruent to $2$ or $3$ mod $5$. By Corollary \ref{cordihed}, the groups $D_{2^2\times 11}$, $D_{2\times 3\times 7}$, $D_{2\times 3\times 11}$, $D_{2\times 7^2}$ and $D_{2\times 7\times 11}$ are not Cay-DS. So, the only possible cases for the factorization of $p-1$ and $p+1$ are $2\times 3$, $2\times 7$, $2\times 11$ and $2\times 11^k$ for some integer $k$. All these cases are impossible, since $p$ is congruent to $2$ or $3$ mod $5$.\\

The center of any Sylow $2$-subgroup of a Suzuki group $\text{Sz}\left( 2^p\right)$, $p$ an odd prime, is elementary abelian of order $2^p$ \cite[p. 182]{Hup2}. It follows from Theorem \ref{2groups} that  $p=3$.  But using {\sf GAP} it is easy to see that $\text{Sz}\left(8\right)$ has the group $D_{26}$ as a subgroup and by Theorem \ref{dihedral}, it is not Cay-DS.\\

The Sylow $2$-subgroup of $\text{PSL}_2\left( 2^p\right)$, $p$ a prime, is an elementary abelian group of order $2^p$ \cite[HauptSatz 8.27, p. 213]{Hup1}. By Theorem \ref{2groups} we have $p=2$ or $p=3$. If $p=2$, then $\text{PSL}_2\left( 2^2\right)$ is isomorphic to $A_5$, the Alternating group of order $60$. If $p=3$, then $\text{PSL}_2\left( 2^3\right)$ has the subgroup $\left( C_2\times C_2\times C_2\right):C_7$. By Lemma \ref{table}, these two latter groups are not DS.\\

The Sylow $2$-subgroup of $\text{PSL}_2\left( 3^p\right)$, $p$ an odd prime, is  an elementary abelian group of order $3^p$ \cite[HauptSatz 8.27, p. 213]{Hup1} which is not a DS by Theorem \ref{3gr} as $3^p\geq 27$. $\hfill\Box$\\

The following result shows that the structure of finite DS groups is somehow restricted. 
\begin{thm}
Let $G$ be a finite DS group of odd order. If   Sylow $3$-subgroups of $G$ are not isomorphic to $C_3 \times C_3$,   then $G$ has the following presentation
$$G=\langle a,b|a^m=1=b^n, b^{-1}ab=a^r\rangle$$
where $r^n\equiv 1 \left(\text{mod}\; m\right)$, $m$ is odd, $0\leq r\leq m$, and $m$ and $n\left( r-1\right)$ are coprime.
\end{thm}

\begin{proof}
By Theorem   \ref{corp}, all  Sylow subgroups are cyclic. Now it follows from Theorem 10.1.10 of \cite{robinson} that
 the group $G$ has a presentation as stated in the theorem. This completes the proof.
\end{proof}

\section*{\bf Acknowledgements}
The research of the first author was in part supported by a grant (No. 93050219) from School of Mathematics, Institute
for Research in Fundamental Sciences (IPM). The research of the first author  is financially supported by the Center of
Excellence for Mathematics, University of Isfahan.


\vspace{1cc}

\noindent Alireza Abdollahi \\
Department of Mathematics,\\
University of Isfahan,\\
Isfahan 81746-73441\\
Iran;\\
and \\
School of Mathematics,\\ Institute for Research in Fundamental Sciences (IPM), \\ P.O. Box 19395-5746, Tehran,\\ Iran \\
E-mail: a.abdollahi@math.ui.ac.ir
\\
\\
\noindent Shahrooz Janbaz \\
Department of Mathematics,\\*
University of Isfahan,\\*
Isfahan 81746-73441\\*
Iran\\*
E-mail: shahrooz.janbaz@sci.ui.ac.ir
\\
\\
\noindent Mojtaba Jazaeri \\
Faculty of Mathematics and Computer Sciences,\\
Shahid Chamran University of Ahvaz,\\
Iran\\
E-mail: seja81@gmail.com
\end{document}